\documentclass{amsart}
  \usepackage{amscd,amssymb,epsfig, epsf}
   \usepackage{epic,eepic}

                     \numberwithin{equation}{subsection}

                     \newtheorem{propo}{Proposition}[section]
                     \newtheorem{corol}[propo]{Corollary}
                  \newtheorem{theor}[propo]{Theorem}
                     \newtheorem{lemma}[propo]{Lemma}
                     \theoremstyle{definition}

                     \theoremstyle{remark}

                     \newcommand{\RR}{\mathbb{R}}

              \newcommand{\card}{\operatorname{card}}
              
                     \newcommand{\id}{\operatorname{id}}

\begin{document}
      \title{Trimming  of finite  metric spaces}
                     \author[Vladimir Turaev]{Vladimir Turaev}
                     \address{%
              Department of Mathematics, \newline
\indent  Indiana University \newline
                     \indent Bloomington IN47405 \newline
                     \indent USA \newline
\indent e-mail: vtouraev@indiana.edu} 
                    \begin{abstract}
We define    a class of trim   metric spaces and  show that each  finite metric space   is  isometric to  the leaf space of a metric forest with trim  base.
\end{abstract}

\maketitle

\section{Introduction}\label{intro}

  Finite metric spaces naturally arise in   mathematics, informatics and phylogenetics, see   \cite{DHM},   \cite{Li}, \cite{SS}.  We introduce   here  a class of trim   metric spaces and  show that every  finite metric space   is    the leaf space of a metric forest with trim   base. We   use the language of pseudometrics and begin by recalling the relevant terminology. 
  A \emph{pseudometric space} is a pair   consisting of a set~$X$ and a mapping $d: X\times X \to \RR $ (the \emph{pseudometric}) such that for all $x,y,z\in X$,   $$d(x,x)=0, \quad d(x,y)=d(y,x)\geq 0, \quad d(x,y)+ d(y,z)\geq d(x,z).$$
Pseudometric spaces $(X,d)$ and $(X',d')$ are \emph{isometric} if there is a bijection $X \to X'$ carrying~$d$ to~$d'$. A map $q:X\to X'$   is \emph{non-expansive} if $d(x,y) \geq d'(q(x), q(y))$ for all $x,y \in X$.  A  \emph{metric space} is a pseudometric space $(X,d)$ such that $d(x,y) >0$ for all  distinct $x, y\in X $ (and then~$d$ is a \emph{metric}).

   Given  points $x,y,z$ of a pseudometric space $ (X,d)$,  we say following K. Menger \cite{Me} that~$x$ lies \emph{between}~$y$ and~$z$  if   $x,y,z$ are pairwise distinct and  \begin{equation}\label{11} d(x,y)+d(x,z)= d(y,z). \end{equation}
We say that a finite pseudometric space $(X,d)$ is   \emph{trim} if  either $\card (X)\leq 1$ or  every point of~$X$ lies between two other   points of~$X$ (that is for each $x\in X$, there are distinct   $y,z\in X\setminus \{x\}$ satisfying \eqref{11}). It is easy to give examples of trim pseudometric spaces. 
For instance, any finite set with $\geq 3$ elements and zero pseudometric $d=0$ is trim. On the other hand, the class of
  trim metric spaces is quite narrow.
In particular, there are no trim   metric spaces  having just two points or three points.
 A
finite subset  of a Euclidean space   with   $  \geq 2$ points and with the induced metric  cannot be trim. Indeed, such a subset   must  contain   a pair of points lying at the maximal distance;   these points cannot lie between   other  points of the subset.   Consequently,    trim  finite metric spaces having more than one point do not   isometrically embed into    Euclidean spaces.

We give here three examples of trim   metric spaces:

  (1)  the set of words of a fixed finite length $\geq 2 $ in a given finite  alphabet  with the Hamming distance defined as the number of positions at which the   letters of two words   differ;

(2) a  subset   of a  Euclidean circle $C\subset \RR^2 $   meeting each   half-circle in $C$   in at least three points; here the distance between   points  is the length of the shorter arc  in $C$ connecting these points;

(3)  the 4-point metric space $ \{a,b,c,e\}$   with metric $d$ defined by
 $$d(a,b)=d(c,e)=r, \,\,\,\, d(a,c)=d(b,e)=s,\,\,\,\,  d(a,e)=d(b,c)=r+s,$$
 where $r,s >0$ are    real numbers. This  is a special case of Example~2 arising from 4-point subsets of~$C$ formed by two pairs of diametrically opposite points.


Our main construction   derives from any finite pseudometric   space~$X$    a trim  finite metric space   $c(X)$, called the \emph{trim core} of~$X$, and  a surjective non-expansive map   $q_X:X\to c(X)$.  The construction  of $c(X)$ and $q_X$ is   functorial:  each isometry of finite pseudometric spaces $\varphi:X\to X'$ induces an isometry $c(\varphi):c(X)\to c(X')$ such that $q_{X'} \varphi= c(\varphi)\,  q_X: X\to c(X')$.
If~$X$ is  a trim metric space, then $c(X)=X$ and $q_X=\id_X$.



    Pseudometric spaces whose trim core  is a point  can be described  in terms of trees.  By a  \emph{tree}, we   mean  a  metric   tree, i.e.,  a connected graph without cycles   whose    edges are  endowed with non-negative real numbers called the  \emph{lengths}.     The set of vertices of a  tree~$\tau$ carries the \emph{path pseudometric}   $  d_\tau $: the distance  between any   vertices  of~$\tau$   is the sum of the lengths of the edges  of~$\tau$  forming the  shortest path between these vertices.
 The \emph{leaf space} of a   tree~$\tau$ is the set of all degree~1 vertices of~$\tau$ together with the     pseudometric~$d_\tau$ restricted to this set. A tree  is  \emph{finite} if it has a finite number of vertices and edges.

\begin{theor}\label{thtree} The trim core of a  finite pseudometric space  $X$ is a   point   if and only if $\card(X)=1$ or $X$   is isometric to the leaf space of a     finite  tree.
  \end{theor}


Besides the trim core, each finite pseudometric space determines   a bunch of rooted trees or, more precisely, a metric forest. 
A \emph{rooted tree} is   a tree    having at least one edge and   a distinguished vertex~$\ast$, the \emph{root}.  A \emph{leaf} of a rooted tree~$\tau$  is a   vertex of~$\tau$ which is distinct from~$\ast$ and has degree 1. The set of all leaves of $\tau$ is denoted by $\partial \tau$; the  pseudometric space $({\partial \tau}  , d_\tau \vert_{{\partial \tau}  })$ is    the \emph{leaf space} of~$\tau$. A \emph{metric forest} $\zeta$  consists of  a metric space $B=(B,d_B)$ called the \emph{base}  and a family of rooted trees $  (\zeta_a,  \ast_a)_{a\in B}$  called the  \emph{components}. By vertices, edges and leaves of~$\zeta$ we  mean the vertices, edges and leaves of the components of~$\zeta$.  The \emph{leaf space} of~$\zeta$   is the disjoint union  $ {\partial \zeta}=\amalg_{a\in B}   \,  {\partial \zeta_a} $    with the following    pseudometric $d_\zeta$:
 for any  $x\in {\partial \zeta_a}  , y\in {\partial \zeta_b}$  with   $a, b \in B$,
\begin{equation}\label{defdzeta}
    d_\zeta(x,y)=\left\{
                \begin{array}{ll}
                  d_{\zeta_a} (x, y) \quad {\text {if}} \,\,   a=b , \\
                  d_{\zeta_a} (x, \ast_a) + d_B(a,b) + d_{\zeta_b} (y, \ast_b ) \quad {\text {if}} \,\,  a \neq b  .
                \end{array}
              \right.
\end{equation}
 For example, any metric space~$B$ determines a metric forest   $  \{[0,1]_a\}_{a\in B}$ where $[0,1]_a$ is a copy of the segment $[0,1]$   with vertices $0,1$, root~$0$, and zero length of the only edge. The leaf space   of this metric forest is isometric to~$B$. 

A metric forest    is  \emph{finite} if its base   is a finite metric space    and all its components are finite rooted trees. 



    \begin{theor}\label{th1} Any finite pseudometric   space~$X$ determines (in a canonical way)  a finite metric forest   with     leaf space~$X$ and base  $c(X)$.
  \end{theor}

For finite pseudometric spaces $X,Y$, we write 
$X \geq Y$ if there is a finite metric forest  with leaf space~$X$ and   base~$Y$.   Theorem~\ref{th1}   implies that  $X \geq c(X)$ for any finite pseudometric   space~$X$. We now explain  that  this property may be used to characterize  $c(X)$ at least up to isometry.

\begin{theor}\label{corol} For any finite metric forest $\zeta$ with trim base~$B$, we have    $c(\partial \zeta) =B$ (up to isometry).
  \end{theor}

\begin{corol}\label{corollll} If a finite pseudometric   space~$X$ and a trim finite metric space~$B$ satisfy  $X\geq B$, then $c(X) =B$ (up to isometry).
  \end{corol}

   We say that two  finite pseudometric spaces are   \emph{trim equivalent} if their trim cores are isometric.
  The next theorem generalizes Theorem~\ref{corol} to arbitrary finite metric forests.

\begin{theor}\label{thh2} The    leaf space and the base of any finite metric forest  are trim equivalent to each other.
  \end{theor}

We  define the trim core in Section~\ref{section0} and prove    Theorems   \ref{thtree}--\ref{thh2} in Sections \ref{section3}--\ref{Proof  of Theorem}.  

  \section{Trimming and the trim core}\label{section0}

  \subsection{Trim spaces}\label{Preliminaries} Given a set~$X$ and a  map  $d:X\times X\to \RR$, we  will  use the same symbol~$d$ for the map  $X\times X\times X \to \RR $ defined by
$$d(x,y,z)=  \frac{ d(x, y)+ d(x, z)-d(y,z)}{2}  $$
 for all $x,y,z \in X$.
We derive from any pseudometric $d: X \times X \to \RR$  a function $\underline d: X \to \RR $:
if~$X$ has  only one  point, then $\underline d =0$; if $X$ has two  points $x,y$, then $\underline d(x)=\underline d(y)= d(x,y)/2$; if~$X$ has $\geq 3$ points, then for all $x \in X$,
$$ {\underline d} (x)= \inf_{y,z \in X\setminus \{x\}, y \neq z }   d(x, y,z)  \geq 0 . $$

We say that a pseudometric space   $(X,d)$ is \emph{trim} if $\underline d=0$. It is clear that for finite~$X$ this definition is equivalent to the one in the introduction.

\subsection{Metric quotient} Any pseudometric space $(X,d)$ determines a metric space $(\widetilde X, \widetilde d)$ called the  \emph{metric quotient} of   $(X,d)$. Here $ \widetilde X = X/{\sim_d }$  where $\sim_d$ is the equivalence relation  on~$X$ defined by $x\sim_d y$ if $d(x,y)=0$ for   $x,y\in X$. The  metric $\widetilde d$ in $\widetilde X$ is uniquely defined by
$ {\widetilde d}(\widetilde x, \widetilde y)=d(x,y)$ where   $x,y$ are any points of~$X$ and $\widetilde x, \widetilde y \in  \widetilde X$ are their equivalence classes.

The  metric quotient of  a trim pseudometric space   may be non-trim. For example, let $X=\{x,y,a,b\}$ and let the distance   between any point of the set $\{x,y\}$ with any point of $\{a,b\}$ be 1 while all the other distances between points of $X$ be~$0$. The resulting   pseudometric space~$X$ is trim  (for instance,   $x$ lies between~$y$ and~$a$). However, the metric space~$\widetilde X$ has   two points and is not trim.

  \subsection{Drift}\label{Drift}\label{notation} Given a set~$X$, a  map  $d:X\times X\to \RR$ and a function $f: X\to \RR$, we define a  map  $d^f:X\times X\to \RR$ by the following rule: for any $x,y \in X$,
  \begin{equation*}
    d^f(x,y)=\left\{
                \begin{array}{ll}
                0 \quad {\text {if}} \,\,  x=y , \\
                  d(x,y)+f(x)+ f(y) \quad {\text {if}} \,\,   x \neq y   .
                \end{array}
              \right.
\end{equation*}
 We say that $d^f$ is obtained from~$d$ by  a \emph{drift}. The idea here is that each point of~$X$ is carried away from all the other points  by the distance $f(x)$.


   \begin{lemma}\label{le49} If $(X,d)$ is a  pseudometric space and   $f:X\to \RR$ is a function such that $f(x)\geq -{\underline d}(x)$ for all $x\in X$, then $(X, {d^f} )$ is a   pseudometric space. Moreover, if $\card(X) \geq 3$ or $\card(X) =2$ and $f$ is a constant function, then   $\underline{d^f}= \underline d+f$.
\end{lemma}

\begin{proof} We first prove that for any distinct $x,y \in X$,
\begin{equation}\label{eqsimple} {\underline d} (x) + {\underline d} (y) \leq d(x,y).  \end{equation}
If $X=\{x,y\}$, then \eqref{eqsimple} follows   from  the definition of~$\underline d$. If $\card(X)\geq 3$, we pick any $z\in X\setminus \{x,y\}$. Then
${\underline d} (x) \leq   { d(x, y ,z)} $
and  ${\underline d} (y) \leq   d(y,x ,z) $.
So,
$${\underline d} (x) + {\underline d} (y) \leq d(x,y,z) + d(y,x,z)= (d(x, y)+ d(y,x))/{2}=d(x,y).$$

We   now check that   ${d^f}:X \times X\to \RR$ is a pseudometric. Clearly,   ${d^f}$ is symmetric and,  by   definition, ${d^f}(x,x)=0$ for all $x\in X$. Formula \eqref{eqsimple} and the assumptions on~$f$ imply that for any distinct $x,y\in X$,
$${d^f}(x,y)=d(x,y)+f(x)+ f(y)\geq {\underline d} (x)  +f(x)+  {\underline d} (y) + f(y) \geq 0.$$
The triangle inequality for $d^f$ may be rewritten as $d^f(x,y,z)\geq 0$
  for any
$x,y,z\in X$.
If $x,y,z$ are pairwise distinct, then
 \begin{equation}\label{ert} d^f(x,y,z)  =  d(x,y,z) +  f(x)  . \end{equation}
 Hence $d^f(x,y,z)   \geq  {\underline d} (x) +  f(x) \geq 0$.
If $x=y$ or $x=z$, then   $d^f(x,y,z)  =0$; if $y=z$, then  $d^f(x,y,z)=d^f(x,y) \geq 0$.

The equality $\underline{d^f} = \underline d  +f $   follows   from \eqref{ert} if $\card(X)\geq 3$ and from the definitions 
if $\card(X)=2$ and $f=const$.
\end{proof}

 \subsection{Trimming}\label{Trimming}
 Given a pseudometric space $(X,d)$, we  can
apply  Lemma~\ref{le49} to  $f=-\underline d:X\to \RR$. This gives a pseudometric   in~$X$   denoted~$d^\bullet$. By definition,
 for any $x,y \in X$,
  \begin{equation}\label{dbulletdef}
    d^\bullet(x,y)=\left\{
                \begin{array}{ll}
                0 \quad {\text {if}} \,\,  x=y  , \\
                  d(x,y)- \underline d(x)- \underline d(y) \quad {\text {if}} \,\,   x \neq y.
                \end{array}
              \right.
\end{equation}
  Lemma~\ref{le49} implies that    $\underline{d^\bullet}=0$ so that      $(X, d^\bullet)$ is a trim pseudometric space. Let
  $ t(X)
= (\widetilde {X} ,  \widetilde {d^\bullet} )$  be   its metric quotient.  We say that   $t(X)$  is obtained from $(X,d)$ by \emph{trimming}. Note that the metric space  $t(X)$   may be non-trim.

By the definition of the metric quotient, $\widetilde X = X/{\sim_{d^\bullet} }$ and
  the projection $p_X:X  \to  \widetilde X  $ is  a   non-expansive surjection.  It is bijective   if and only if    ${d^\bullet}$   is a metric. In this case,  $\widetilde {X}=X$, $  \widetilde {d^\bullet}={d^\bullet}$   and   $t(X) $ is a trim metric space.


  Though it is  not used in the sequel, we     compute the equivalence relation  ${\sim_{d^\bullet} }$ in~$X$ directly  from~$d$. Note first  that if $1 \leq \card(X) \leq 3$, then   all elements of~$X$ are related by ${\sim_{d^\bullet} }$ and $t(X)=\{pt\}$ (exercise). 

  \begin{lemma}\label{leLAST}  Suppose that $\card(X) \geq 3$. Two distinct elements $x,y\in X$ are related by ${\sim_{d^\bullet} }$ if and only if
  for any distinct $x', x''\in X\setminus \{x\}$ and any   distinct $y', y''\in X\setminus \{y\}$, we have
  $$2 d(x,y) + d(  x',x'') +   d(  y',y'')  \leq d(x, x')+ d(x, x'') + d(y, y')+ d(y, y'')    .$$
  \end{lemma}

 \begin{proof}   It follows from the definitions that  $$x {\sim_{d^\bullet} } y  \Longleftrightarrow  d^\bullet (x,y)=0 \Longleftrightarrow  d(x,y)= \underline d (x)+ \underline d (y).$$ By~\eqref{eqsimple}, the latter equality holds iff
 $d(x,y) \leq \underline d (x)+ \underline d (y)$. Recall that
 $${\underline d} (x)= \inf_{x',x'' \in X\setminus \{x\}, x' \neq x'' }   d(x, x',x'') \,\,\, {\text {and}} \,\,\, {\underline d} (y)= \inf_{y',y'' \in X\setminus \{y\}, y' \neq y'' }   d(y, y',y'').$$
 Thus,
  $$x{\sim_{d^\bullet} }  y \Longleftrightarrow  d(x,y) \leq d(x, x',x'') +   d(y, y',y'') $$
  for any distinct $x', x''\in X\setminus \{x\}$ and     distinct $y', y''\in X\setminus \{y\}$. Substituting here the defining expression of the function $d:X\times X \times X\to \RR$, we obtain the claim of the lemma.
 \end{proof}

 \subsection{The trim core}\label{The transformation}\label{The trim core} Starting from a  pseudometric space~$X$ and iterating the trimming,   we  obtain   metric spaces $\{t^n(X)\}_{n\geq 1}$ and non-expansive surjections
   \begin{equation}\label{line} X =t^0(X) \stackrel{p_0}{\longrightarrow} t(X) \stackrel{p_1}{\longrightarrow} t^2(X) \stackrel{p_2}{\longrightarrow}  \cdots  \end{equation}  where     $p_n=p_{t^{n}(X)}  $ for all $n\geq 0$.   We say that~$X $ has   \emph{finite height} if    $t^n(X)$ is a  trim metric space  for some $n\geq 0$.
 The minimal such~$n$ is  the \emph{height} of~$X$ and   the corresponding metric space $c(X)=t^n(X)$  is the \emph{trim core} of~$X$.
 For a  pseudometric space $X$ of finite height $n\geq 0$, the map
  $$q_X=p_{n-1}\circ \cdots p_1 \circ p_0: X\to t^n(X)= c(X) $$
  is    a non-expansive surjection while  the maps $\{p_m \}_{m \geq n} $ are isometries. Note  that   $c(t(X))=c(X)$ and, if $n\geq 1$, then $height(t(X)) = n-1$.
Clearly,  $height(X)=0$ if and only if $X$ is a trim metric space     and then $c(X)=t(X)=X$.

These   definitions and  results    apply, in particular,   to  any finite pseudometric space~$X $.    Indeed,   for $ N=\card(X)  $,  at least one of the surjections $p_0, p_1,..., p_{N-1}$  in \eqref{line}   must be   bijective. Hence, 
$height (X )\leq N$  and   $c(X)=t^N(X)$.

 \subsection{Remarks}\label{example}   1. If a pseudometric space   $ (X,d)$ has  $\geq 3$ points, then the function ${\underline d}:X \to \RR$   can be computed by a slightly shorter formula
 $${\underline d} (x)= \inf_{y,z \in X\setminus \{x\} }    d(x, y ,z)  $$
 for any $x \in X$. 
To deduce this    from the definition of $\underline d$, it suffices to note that by \eqref{eqsimple}, we have  ${\underline d}(x) \leq d(x,y)=d(x, y ,y) $
 for all $ y\in X\setminus \{x\}$.

2.  The  maps  \eqref{line} induce a sequence of equivalence relations $\{\sim_n\}_{n\geq 1}$ in any pseudometric space $(X,d)$:   points of $  X$ are related by $\sim_n$ if their images in $t^n(X)$ are equal.  Clearly, ${\sim_{1} }={\sim_{d^\bullet} }$ and $x\sim_n y \Longrightarrow x\sim_{n+1} y$ for all $x,y \in X$.


3.    For any $X,d,f $ as in   Lemma~\ref{le49}, the pseudometric space $X^f=(X, d^f)$ satisfies $(d^f)^\bullet=d^\bullet :X \times X \to \RR$. Therefore $t(X^f)=t(X)$, i.e., the trimming turns $X$ and $X^f$ into the same metric space.  Consequently,
if $X $ has a finite height, then so does $X^f $ and
$c(X^f)=c(X )$. 

4. Consider a finite set $X$ and a  map $h:X\times X \to \RR$ such that $h(x,y)=h(y,x)$ for all $x,y\in X$.  For $r\in \RR$, we define a  map $h_r:X\times X\to \RR$   by
\begin{equation*}
    h_r (x,y)=\left\{
                \begin{array}{ll}
                0 \quad {\text {if}} \,\,  x=y  , \\
                  h(x,y)+r \quad {\text {if}} \,\,   x \neq y.
                \end{array}
              \right.
              \end{equation*}
              It is clear that  $h_r$ is a pseudometric in~$X$ for all sufficiently big~$r$.
 The  metric space $t(X,h)=t(X,h_r)$   is easily seen to be independent of the choice of~$r$.
 This defines the trimming of the pair  $(X,h)$ and allows us to define the trim core of this pair    to be the trim metric space $c(t(X,h))$.

   \section{Proof  of Theorem~\ref{th1}}\label{section3}

 \subsection{Lemmas} We begin with  two lemmas.

 \begin{lemma}\label{th3}  Any  pseudometric space $X=(X,d)$  determines for all $n\geq 1$  a   metric forest   with base  $t^n(X)$ and    leaf space~$X$.
  \end{lemma}

 \begin{proof} 
Set $X_0=X$ and   $d_0=d$.
For  $i\geq 1$, let $X_i$ and $d_i:X_i\times X_i \to \RR $ be the underlying set and the metric of the metric space $t^i(X)$. For all $i\geq 0$, consider
  the   function $\underline{d_i} :X_i \to \RR $   and the map $p_i=p_{t^i(X)}: X_i \to X_{i+1}$.
  We  fix $n\geq 1$ and  form a graph  whose  set of  vertices  is  the disjoint union $\amalg_{i=0}^{n} X_i$. Each    $v\in X_i$ with   $i \leq n-1$ is connected  to   $p_i(v)\in X_{i+1}$ by an edge $e_v$ of length $\underline{d_i} (v)\geq 0$. It is clear that the resulting graph, $\zeta$,    is a disjoint union of trees, and each of these trees has precisely one vertex belonging to  $X_{n}$. We take this vertex as the root. Then  every point   of $ X_{n}$ is the root of a unique tree component   of~$\zeta$. Thus,~$\zeta$ is a metric forest with base $X_{n} $. Note that the vertices of $\zeta$ lying in $X_0=X$ have degree $1$.
 The surjectivity of    $p_0, p_1,..., p_{n-2}$ implies that the vertices of~$\zeta$ lying in $X_1\amalg X_2 \amalg  \cdots \amalg X_{n-1}$ have degree $\geq 2$. Thus,   $\partial \zeta= X$.
 We need only  to show that the  path pseudometric $d_\zeta$ in $\partial \zeta $   coincides with the original metric~$d$ in~$X$.

 Observe   that any    $x\in X $ determines a sequence of  points   $\{x_i \in X_i\}_{i=0}^{n}$     by  $x_0=x$ and $x_{i+1}=p_i(x_i)$ for   $i\leq n-1$.  Clearly,~$x$ is a degree 1 vertex of the tree $\zeta_{x_{n}}$ with root $x_{n}$.
We   pick   any    $x,y\in X$ and prove that $d_\zeta(x,y)= d(x,y)$. It suffices to handle the case $x \neq y$. Suppose first that $x_{n}= y_{n}$ so that   $x$ and  $y$ are vertices of the same tree  $\mu=\zeta_{x_{n}}$.
   Let $1 \leq m \leq n $ be the smallest   index such that $x_{m}=y_{m}$.
The shortest path from $x$ to $y$ in $\mu$ is formed by the   edges $$ e_{x_0 }, e_{x_1 },...,  e_{x_{m-1} }, e_{y_{m-1}},  ..., e_{y_{1}}, e_{y_{0}}.$$
  By~\eqref{defdzeta}  and the definition of the path pseudometric $d_\mu$ in~$\mu $,
  \begin{equation}\label{almot} d_\zeta(x,y)=d_{\mu} (x,  y) =\sum_{i=0}^{m-1} (\underline{d_i} (x_i )+ \underline{d_i} (y_i )) .\end{equation}
 By the definition of the metric $d_{i+1}$ in $X_{i+1}$, for all  $i=0,1,...,m-1$, we have
 \begin{equation}\label{almotr} \underline{d_i} (x_i ) +\underline{d_i} (y_i )=d_{i}(x_{i},y_{i})- d_{i+1}(x_{i+1},y_{i+1}) .\end{equation}
   Substituting these expressions in \eqref{almot}, we obtain
  $$d_\zeta(x,y)=d_{0}(x_{0},y_{0})- d_{m}(x_{m},y_{m})=d_0(x_0,y_0)=d(x,y).$$

  Suppose now that $x_{n}\neq y_{n}$ so that   $x$, $y$ lie in different trees $\mu=\zeta_{x_{n}}$, $\eta=\zeta_{y_{n}}$ with roots respectively $x_{n}$ and $y_{n}$.
The shortest path from $x$ to $x_{n}$ in $\mu$ is formed by the   edges $\{ e_{x_i }\}_{i=0}^{n-1}$ and similarly for~$y$.
Hence,
$$ d_{\mu} (x,  x_{n})=\sum_{i=0}^{n-1} \underline{d_i} (x_i ) \quad {\rm {and}} \quad  d_{\eta} (y,  y_{n})=\sum_{i=0}^{n-1} \underline{d_i} (y_i ) $$
where    $d_{\mu}, d_{\eta}$ are the path pseudometrics in    $\mu, \eta$, respectively.
By~\eqref{defdzeta},
$$
    d_\zeta(x,y)=
                  d_{\mu} (x,  x_{n}) + d_{n}(x_{n}, y_{n}) + d_{\eta} ( y, y_{n} )
$$
$$  =d_{n}(x_{n}, y_{n})+
                 \sum_{i=0}^{n-1} (\underline{d_i} (x_i ) +\underline{d_i} (y_i ))  .$$
Substituting     \eqref{almotr}   for  all $i$, we obtain
  $d_\zeta(x,y)=d_0 (x_0, y_0)= d(x,y)$.
\end{proof}

   \begin{lemma}\label{th8} Any  pseudometric space $X$ of finite height determines (in a canonical way)  a  metric forest   with base  $c(X)$ and   leaf space~$X$.
  \end{lemma}

 \begin{proof} Set $n=height(X) \geq 0$. If   $n=0$,   then  $c(X)=X$ and the   metric forest $ \{[0,1]_a\}_{a\in X}$ described in the introduction satisfies the conditions of the lemma. If     $n \geq 1$, then this lemma is a special case of Lemma~\ref{th3}.  \end{proof}

 \subsection{Proof  of Theorem \ref{th1}.}     Theorem \ref{th1} is a direct consequence of Lemma \ref{th8} since all finite metric spaces have finite height.

  \section{Proof  of Theorem~\ref{thtree}}\label{section3+}

  \subsection{Lemmas}  Given a (non-rooted) tree $\tau$, we let $\tau^\bullet$ be the tree  obtained from~$\tau$  by changing to zero the lengths  of all edges adjacent to degree~1 vertices   while keeping the lengths of all other edges. 

   \begin{lemma}\label{cloleththree} Let $\tau$ be a finite non-rooted tree  with leaf space $(X, d=d_\tau\vert_X)$. If~$\tau$ has no vertices of degree 2 and $\card (X)\geq 3$, then
   the   pseudometric   $  d^\bullet $ in~$X$ defined by \eqref{dbulletdef} coincides with   the  path pseudometric of the tree~$\tau^\bullet$.
  \end{lemma}

 \begin{proof} We  first compute the function $\underline d:X\to \RR$ determined by the pseudometric $d=d_\tau$ in $X$.     Pick any    $x\in X $ and let $ e_x $ be the   unique edge of~$\tau$ adjacent to~$x$. Let $L\geq 0$ be the length of~$e_x$. We claim that    $ \underline{d}(x)=L$.  To see it, we compute  $d (x,y,z)  \in \RR$ for any     $y,z \in X\setminus \{x\}$. Since $\tau$ is a tree,  there   are injective paths $p_x, p_y, p_z$ in $ \tau$ leading respectively from $x,y,z$ to a vertex~$v$ of   $\tau$
   and meeting solely in~$v$.
   Denoting by $\ell(p)$ the length of a path $p$, we obtain
    $$d(x, y)=d_{\tau}(x, y)= \ell(p_x)+\ell(p_y),$$
    $$d(x, z)=d_{\tau}(x, z)= \ell(p_x)+\ell(p_z),$$
    $$d(y,z)=d_{\tau}(y,z)= \ell(p_y)+\ell(p_z).$$
   Consequently,  $$ d(x,y,z)=  \frac{ d(x, y)+ d(x, z)-d(y,z)}{2} =\ell(p_x) .$$
  Since  the path $p_x$ has to traverse the   edge $e_x$, we have $\ell(p_x) \geq L$. Thus, $d (x,y,z) \geq L$.
  To prove that $ \underline{d}(x)=L$, it remains   to find   distinct   $y,z\in X \setminus \{x\}$ such that the   path $p_x$   above consists only of the edge $e_x $ so that $d(x,y,z)=\ell(p_x)=L$. Such   $y,z$ do exist because the  second endpoint of  $e_x $  has degree   $\geq 3$. The latter holds because~$\tau$ has no vertices of degree 2 and $\card(X) \geq 3$.

    The computation of~$\underline d$ above and  the definition of     $  d^\bullet$     imply that  $d^\bullet$ is the  path pseudometric of~$\tau^\bullet$.
 \end{proof}

  \begin{lemma}\label{leththree} For any finite  pseudometric space~$X$, the following three conditions are equivalent:

  (i) $c(X)=\{pt\}$;

    (ii) $X $ is   the leaf space of a finite rooted tree;

    (iii) $\card(X)=1$ or $X$   is   the leaf space of a     finite   non-rooted  tree.
  \end{lemma}

 \begin{proof} The implication  (i) $\Longrightarrow$ (ii)    follows from Theorem~\ref{th1}. To prove the  implication
   (ii) $\Longrightarrow$ (iii), suppose that~$X $ is   the leaf space of a finite rooted tree~$ \tau $ with root~$\ast$.  If
 the degree of~$\ast$ is   $\neq 1$,   then the   leaf spaces of~$\tau$ viewed   as a rooted tree and a non-rooted tree coincide and we get (iii). If~$\ast$
 has degree~1, then~$\ast$ is adjacent to a unique edge~$e$   connecting $\ast$ to another vertex, $v$,   of~$\tau$. If~$v$ has degree 1, then~$\tau$ has no edges other than $e$ and $X=\{v\}$, so that $\card(X)=1$. If~$v$ has degree $\geq 2$, then~$X$  is the leaf space of the rooted tree  obtained from~$\tau$ by deleting~$\ast$ and~$e$  and taking~$v$ for the   root. This new rooted tree has less edges than~$\tau$ and the  implication
   (ii) $\Longrightarrow$ (iii)  follows by induction.

 We now prove the implication   (iii) $\Longrightarrow$ (i). If $\card (X)=1$, then $c(X)=X=\{pt\}$.
 Suppose now that  $\card (X)\geq 2$ and~$X $  is the leaf space
 of a  finite  non-rooted  tree~$\tau$.
 If $\tau$  has a vertex of degree 2, then eliminating this vertex and uniting its   adjacent edges into a single edge (the lengths add  up)  we obtain a new   tree with the same leaf space. Similarly, if $\tau$  has an edge of length zero with both endpoints of degree $\neq 1$, then contracting this edge,  we obtain  a new   tree with the same leaf space.
 Thus, without loss of generality we can assume that
  $\tau$ has no vertices of degree 2 and
 the length of each edge of~$\tau$ with both endpoints of degree $\neq 1$ is positive.
We prove that $c(X)=\{pt\}$   by induction on $\card (X)$. If $\card (X)= 2$, then  $c(X)=t(X)=\{pt\}$. If $\card (X)\geq 3$, then by Lemma~\ref{cloleththree}, the   pseudometric   $  d^\bullet  $ in~$X$ coincides with the  path pseudometric of the tree~$\tau^\bullet$.  The assumptions on~$\tau$ imply that the equivalence relation $\sim_{d^\bullet}$ in~$ X$ relates precisely  those degree~1 vertices of $\tau$   whose adjacent edges share the   second endpoint. This equivalence relation is non-trivial:  an easy induction on the number of edges shows that a   finite tree  having at least three vertices  and no   vertices of degree~2   must have distinct   degree~1 vertices whose adjacent edges share an endpoint (cf.\ Lemma~\ref{lemma54} below).  Hence the   set $\widetilde {X} = X/\sim_{d^\bullet}$ has   less elements than~$X$. 
  The  pseudometric space $t(X) =(\widetilde {X} , \widetilde {d^\bullet} )$   is the leaf space of the tree  obtained from $\tau^\bullet$ by   deleting all but one   vertices (and the adjacent  edges) in each equivalence class in $ X$.    By the definition of the  trim core and  the induction assumption, $c(X)=c(t(X)) =\{pt\}$.
 \end{proof}

 \subsection{Proof  of   Theorem \ref{thtree}.} Theorem \ref{thtree}   follows from the equivalence of the conditions (i) and (iii) in Lemma~\ref{leththree}.

  \section{Proof  of Theorems \ref{corol} and~\ref{thh2}}\label{Proof  of Theorem}

  \subsection{Conventions and lemmas.}  Throughout Section~\ref{Proof  of Theorem}  all trees and   forests are assumed to be finite.   Every leaf~$x$ of a rooted tree~$\tau$ is adjacent to a single edge   of~$\tau$ whose other vertex is called the \emph{neighbor} of~$x$  (the neighbor   may  happen to be the root of~$\tau$ but cannot be a leaf of~$\tau$). Two leaves of a rooted tree are \emph{contiguous} is they have the same neighbor.
   We   say that a rooted   tree  is \emph{reduced} if
it has no vertices of degree~2 except possibly the root. The number  of edges of a  tree~$\tau$ is denoted by $ \vert \tau \vert$.

  \begin{lemma}\label{lemma54} Any reduced rooted tree~$\tau$ with $ \vert \tau \vert \geq 2$  has at least one pair of contiguous leaves.
  \end{lemma}

  \begin{proof}    We proceed by induction on   $ \vert \tau \vert  $.
   If $\vert \tau \vert=2$, then $\tau$ has a vertex of degree~2 which has to be   the root because~$\tau$ is reduced; the other two vertices of~$\tau$ are contiguous leaves. The induction step goes as follows. Suppose that $\vert \tau \vert \geq 3$ and let~$\ast$ be the root of~$\tau$. If~$\ast$  has degree~$1$, then eliminating $\ast$ and the   adjacent edge, and taking the second vertex of this edge as the new root, we obtain a reduced rooted tree~$\tau'$ with $\vert \tau' \vert=\vert \tau \vert-1 \geq 2 $.   By the induction assumption, $\tau'$ has   contiguous leaves. The same leaves are contiguous in~$\tau$.
  If~$\ast$   has degree $m\geq 2$, then~$\tau$ is a union of~$m$ reduced rooted trees $\{\tau_i\}_{i=1}^m$ meeting   in the  common root~$\ast$. Clearly,  $\vert \tau_i \vert <\vert \tau \vert$ for all~$i$ and either $\vert \tau_i \vert =1$ for all~$i$ or there is an $i  $ such that  $\vert \tau_i \vert \geq 2$.  In the former case, the  vertices of  $\{\tau_i\}_{i=1}^m$ distinct from~$\ast$ are contiguous leaves of~$\tau$. If $\vert \tau_i \vert \geq 2$,  then by the induction assumption, $\tau_i$ has      contiguous leaves. The same leaves are contiguous in~$\tau$.
  \end{proof}
  
  The following claim is a version of Lemma~\ref{cloleththree} for metric forests. 

   \begin{lemma}\label{lemma42enm}\label{lemma43enm} Let  $ \zeta$ be   a   metric forest  whose components are reduced rooted trees and whose  base $B=(B,d_B)$ is a trim metric space with $\geq 2$ points. 
    Consider the leaf space $(\partial \zeta, d=d_\zeta)$ of $\zeta$.  Then the pseudometric space $(\partial \zeta, d^\bullet) $ is the leaf space of  the metric forest $\zeta^\bullet$  obtained from $\zeta$
   by changing to zero the lengths  of all edges adjacent to   leaves (and keeping the lengths of all the other edges).  
  \end{lemma}

    \begin{proof} We   compute the function $\underline{d}: \partial \zeta \to \RR $. Pick any $x\in  {\partial \zeta_a } \subset \partial \zeta $ where $a\in B$.   Let $e_x$ be the unique edge   of   $\zeta_{a}$ adjacent to~$x$. Let~$L$ be the length  of   $e_x$. We claim that $\underline d(x)=L$.  We first show
   that $d (x,y,z) \geq L$ for all   $y,z \in \partial \zeta\setminus \{x\}$.  

    Case 1:  $ y,z\in \partial \zeta_a$. In this case the computations in the proof of Lemma~\ref{cloleththree} (for $\tau=\zeta_a$) apply and give $d (x,y,z) \geq L$.


    Case 2:  $ y \in \partial \zeta_a$ and $z \in \partial \zeta_b$ with $b\in B\setminus \{a\}$. Then  there are injective paths $p_x, p_y, p_{\ast }$ in $ \zeta_a$ leading respectively from $x,y,\ast_a$ to
   a vertex $v$ of   $\zeta_a$  and meeting solely in~$v$. Then
    $$d(x, y)=d_{\zeta_a}(x, y)= \ell(p_x)+\ell(p_y),$$
    $$d(x, z)=d_{\zeta}(x, z)= \ell(p_x)+\ell(p_{\ast }) +d_B(a, b)+  d_{\zeta_b}( z, \ast_b ),$$
    $$d(y, z)=d_{\zeta}(y, z)= \ell(p_y)+\ell(p_{\ast }) +d_B(a, b)+  d_{\zeta_b}( z, \ast_b ).$$
   Consequently,  $$ d(x,y,z)=  \frac{ d(x, y)+ d(x, z)-d(y,z)}{2} =\ell(p_x)  \geq L.$$ 

    Case 3:  $y \in \partial \zeta_b$ with $b\in B\setminus \{a\}$ and $ z \in \partial \zeta_a$. This case  is similar  to Case 2.

     Case 4:   $y, z \in \partial \zeta_b$ with $b\in B\setminus \{a\}$. Then $d(y, z)=d_{\zeta_b}(y, z)$ and
    $$d(x, y)=d_{\zeta}(x, y)=d_{\zeta_a}(x, \ast_a) +d_B(a, b)+  d_{\zeta_b}( y, \ast_b ),$$
    $$d(x, z)=d_{\zeta}(x, z)= d_{\zeta_a}(x, \ast_a) +d_B(a, b)+  d_{\zeta_b}( z, \ast_b ).$$
  Using the triangle inequality for  $d_{\zeta_b}$, we deduce that
 $ d(x,y,z)  \geq d_{\zeta_a}(x, \ast_a) \geq L$.

  Case 5:  $y  \in \partial \zeta_b$ and $z\in   \partial \zeta_c$ with distinct $b,c \in B\setminus \{a\}$. Then
    $$d(x, y)=d_{\zeta}(x, y)=d_{\zeta_a}(x, \ast_a) +d_B(a, b)+  d_{\zeta_b}( y, \ast_b ),$$
    $$d(x, z)=d_{\zeta}(x, z)= d_{\zeta_a}(x, \ast_a) +d_B( a,  c)+  d_{\zeta_c}( z, \ast_c ),$$
     $$d(y, z)=d_{\zeta}(y, z)= d_{\zeta_b}(y, \ast_b) +d_B( b,  c)+  d_{\zeta_c}( z, \ast_c ).$$
    Therefore
   $$
   d(x,y,z)= d_{\zeta_a}(x, \ast_a)+ \frac{ d_B(a,b)+ d_B(a,c)-d_B(b,c)}{2} \geq d_{\zeta_a}(x, \ast_a) \geq L.
  $$

     To prove that $ \underline{d}(x)=L$ it remains  to exhibit distinct  $y,z\in \partial \zeta\setminus \{x\}$ such that $d(x,y,z)=L$. 
     Suppose first that   the edge $e_x$ connects~$x$ to  $ \ast_a$. Since $B$ is   trim, $a$ lies between certain points $b,c \in B  $. Pick any $y\in \partial \zeta_b$ and $z\in \partial \zeta_c$. The computation in Case~5    yields  $d(x,y,z)= d_{\zeta_a}(x, \ast_a) = L$. Suppose now that $e_x$ connects~$x$ to a vertex $v \neq \ast_a$ of $\zeta_a$.
      Since  $\zeta_a$ is reduced,    $degree (v) \geq 3$. This   implies the existence of a leaf $y\in \partial \zeta_a \setminus \{x\}$  such that certain  injective paths $  p_y, p_{\ast }$ in $ \zeta_a$ lead  respectively from $ y,\ast_a$ to~$v$   and meet  solely in~$v$.  Pick    any $b\in B\setminus \{a\}$ and $z\in \partial \zeta_b$. The computation  in Case~2 (with $p_x=e_x$) shows that $d(x,y,z)=\ell(p_x) =L$. 

 The computation of~$\underline d$ above and  the definition of     $  d^\bullet$     imply that  $d^\bullet$ is the  pseudometric in $\partial \zeta= \partial \zeta^\bullet$ determined by the metric forest $\zeta^\bullet$. \end{proof}


   \begin{lemma}\label{lemma42e} Let~$\zeta$ be a   metric forest whose  base~$B $ is a trim metric space with $\geq 2$ points. Then at least one of the following  two claims hold:
    
   (i) the leaf space $\partial \zeta=(\partial \zeta, d_\zeta)$ of~$\zeta$ is obtained from~$B$ by a drift;
      
    (ii)    there is a   metric forest with less edges than~$\zeta$,  same base~$B$ as~$\zeta$  and the   leaf space   $\partial \zeta$ or $t(\partial \zeta)$.
  \end{lemma}


   \begin{proof} Denote  by $\vert \zeta \vert$ the number of edges of~$\zeta$. 
  The definition of a metric forest implies that $\vert \zeta \vert \geq  \card(B)$.   If     $\vert \zeta \vert= \card(B)$, then each    tree $\zeta_a$  with $a\in B$  is  just an edge   with two vertices (one of them being  the root) and a certain length   $f(a)\geq 0$. This defines a function $f:B\to \RR $ such that  $ \partial \zeta=B^f$. So,    (i) holds.

   We    say that a metric forest  is    \emph{reduced} if   all its components are reduced rooted trees   having no edges  of length zero with both endpoints of degree $\neq 1$. If~$\zeta$ is not reduced, then eliminating vertices of degree 2 and contracting    edges
  as in the proof  of   Lemma \ref{leththree}, we   obtain    a     metric forest    satisfying  (ii).  

  It remains to prove (ii) in the case where  $\vert \zeta \vert >  \card(B)$ and $\zeta$ is   reduced.  Since  all   components of~$\zeta$ are reduced, Lemma~\ref{lemma42enm}  implies that   the pseudometric space $(\partial \zeta , d^\bullet)$ (where   $d=d_\zeta$) is the leaf space of the metric forest $\zeta^\bullet$
   obtained from~$\zeta$
   by changing to zero the lengths  of all edges adjacent to   leaves. Since $\zeta$ is reduced, the lengths of all other edges of~$\zeta^\bullet$ are non-zero.  This implies that two distinct leaves $u,v  $ of~$\zeta$ are related by the equivalence relation $\sim_{d^\bullet}$ (i.e., satisfy $d^\bullet (u,v)=0$) if and only if     they are contiguous leaves of  a tree component of~$\zeta$. Then the   metric space $t(\partial \zeta,d_\zeta) =(\widetilde {\partial \zeta} , \widetilde {d^\bullet} )$   is the leaf space of the metric forest $\zeta'$   obtained from $\zeta^\bullet$ by   deleting all but one   vertices (and the adjacent  edges) in each equivalence class of leaves.
     To check that $\vert \zeta'\vert <\vert \zeta \vert$, it suffices to note that the equivalence relation $\sim_{d^\bullet}$ in $\partial \zeta$ is non-trivial. Indeed,  the condition $\vert
 \zeta \vert  > \card (B)$ implies that~$ \zeta $ has a  tree  component with $\geq 2$ edges and     by Lemma~\ref{lemma54}  such a component has   contiguous leaves.      \end{proof}


%



 \subsection{Proof  of Theorem \ref{corol}.}  If $B=\{pt\}$, then $\zeta$ is    a   rooted tree and by Lemma~\ref{leththree}, $c(\partial \zeta)=\{pt\}=B$. If $\card(B)\geq 2$, then applying Lemma~\ref{lemma42e} recursively, we obtain that for a sufficiently big integer $N\geq 1$, the metric space       $t^N(\partial \zeta) $  is obtained from~$B$ by a drift. By Remark~\ref{example}.3,   $$t^{N+1}(\partial \zeta)=t(t^N(\partial \zeta))= t(B)   .$$ Since $B$ is trim, $t(B)=B$. Thus, $t^{N+1}(\partial \zeta)=B$ is trim and  $c(\partial \zeta) =B$.
 
 \subsection{Proof  of Theorem \ref{thh2}.} We first prove the transitivity of   the relation $\geq$ between finite pseudometric spaces (cf.\ the introduction).
Consider      metric forests $\zeta=\{\zeta_a\}_{a\in A}$    and  $\eta=\{\eta_b\}_{b\in B}$ with bases $A$ and $B$ respectively. We claim that if    $A=\partial \eta$, then   $\partial \zeta \geq B$. To see it,  for each $b\in B$    glue the trees $\eta_b $ and $ \{\zeta_a\}_{a }$ where $a\in A$ runs over the leaves of $\eta_b $. The   gluing goes by identifying each such~$a$  with the root of   $\zeta_a$ and produces a   tree $\mu_b$. We take     the root of $\eta_b$ as the root of $\mu_b$.   Clearly, the   metric forest $ \{\mu_b\}_{b\in B}$ has the same leaf space as~$   \zeta$. Hence, $\partial \zeta \geq B$.
  
  To prove Theorem \ref{thh2}, consider a   metric forest~$\zeta$ with base~$A$.  Then  $\partial \zeta\geq A$ and, by Theorem \ref{th1}, $A \geq c(A)$. Therefore $\partial \zeta\geq    c(A)$, so that there is a metric forest with leaf space   $\partial \zeta$ and   base  $c(A)$. Since $c(A)$ is trim,  Theorem~\ref{corol} implies that $c(\partial \zeta)  =c(A)$. Thus, $\partial \zeta$ is trim equivalent to~$A$.


\begin{thebibliography}{CJKLS}

%

%
%



\bibitem[DHM]{DHM} A. Dress, K. T. Huber, V. Moulton,
Metric Spaces
in Pure and Applied Mathematics. Documenta Mathematica,
Proc. Conf. \lq\lq Quadratic Forms and Related Topics", LSU, Baton Rouge, 2001, 121--139.

 \bibitem[Li]{Li}  N. Linial,
Finite metric spaces - combinatorics, geometry and algorithms.   Proc. Int. Congress of Mathematicians, Vol. III (Beijing, 2002), 573--586, Higher Ed. Press, Beijing, 2002.





\bibitem[Me]{Me} K. Menger, {\it Untersuchungen ¨uber allgemeine metrik,}  Math. Ann. 100 (1928), 75--163.

\bibitem[SS]{SS} C. Semple, M. Steel, Phylogenetics, Oxford Lecture Series in Mathematics and its Applications,
vol. 24, Oxford University Press, Oxford, 2003.



                     \end{thebibliography}
                     \end{document}